\documentclass{article}
\usepackage{amssymb}
\usepackage{amsfonts}
\usepackage{amsmath}

\newtheorem{theorem}{Theorem}

\newtheorem{corollary}[theorem]{Corollary}

\newtheorem{definition}[theorem]{Definition}

\newenvironment{proof}[1][Proof]{\noindent\textbf{#1.} }{\ \rule{0.5em}{0.5em}}
\input{tcilatex}
\begin{document}

\title{Chebyshev polynomial coefficient bounds for a subclass of
bi-univalent functions}
\author{\c{S}ahsene Alt\i nkaya$^{1,\ast }$ and Sibel Yal\c{c}\i n$^{1}$ \\
%EndAName
$^{1}$Department of Mathematics, Faculty of Arts and Science,\\
Uludag University, Bursa, Turkey.\\
$^{\ast }$corresponding author. e-mail: sahsene@uludag.edu.tr\\
$^{1}$syalcin@uludag.edu.tr}
\maketitle

\begin{abstract}
In this work, considering\ a general subclass of bi-univalent functions and
using the Chebyshev polynomials, we obtain coefficient expansions for
functions in this class.


Mathematics Subject Classification, 2010: 30C45, 30C50.
\end{abstract}

\section{Introduction and Definitions}

Let $D$ be the unit disk $\left\{ z:z\in 
%TCIMACRO{\U{2102} }%
%BeginExpansion
\mathbb{C}
%EndExpansion
\text{ and }\left\vert z\right\vert <1\right\} \ ,$ $A$ be the class of
functions analytic in $D,$ satisfying the conditions%
\begin{equation*}
f(0)=0\text{ and}\ \ \ f^{\prime }(0)=1.
\end{equation*}%
Then each function $f$ in $A$ has the Taylor expansion 
\begin{equation}
f(z)=z+\overset{\infty }{\underset{n=2}{\sum }}a_{n}z^{n}.  \label{eq1}
\end{equation}%
Further, by $S$ we shall denote the class of all functions in $A$ which are
univalent in $D.$

The Koebe one-quarter theorem \cite{Duren 83} ensures that the image of $D$
under every function $f$ from $S$ contains a disk of radius $\frac{1}{4}.$
Thus every such univalent function has an inverse $f^{-1}$ which satisfies%
\begin{equation*}
f^{-1}\left( f\left( z\right) \right) =z~,~\left( z\in D\right)
\end{equation*}%
and 
\begin{equation*}
f\left( f^{-1}\left( w\right) \right) =w~,~\left( \left\vert w\right\vert
<r_{0}\left( f\right) ~,~r_{0}\left( f\right) \geq \frac{1}{4}\right) ,
\end{equation*}%
where%
\begin{equation*}
f^{-1}\left( w\right) =w~-a_{2}w^{2}+\left( 2a_{2}^{2}-a_{3}\right)
w^{3}-\left( 5a_{2}^{3}-5a_{2}a_{3}+a_{4}\right) w^{4}+\cdots .
\end{equation*}

A function $f\in A$ is said to be bi-univalent in $D$ if both $f$ and $%
f^{-1} $ are univalent in $D.$ Let $\Sigma $ represent the class of
bi-univalent functions in $D$ given by (\ref{eq1}). For some intriguing
examples of functions and characterization of the class $\Sigma ,$ one could
refer Srivastava et al. \cite{Srivastava 2010}, and the references stated
therein (see also, \cite{Owa 2012}). Recently there has been triggering ,
interest to study the bi-univalent functions class $\Sigma $ (see \cite%
{Altinkaya and Yalcin 2014}, \cite{Crisan 2013}, \cite{Keerthi and Raja 2013}%
, \cite{Magesh and Yamini 2013}, \cite{Porwal and Darus 2013}) and obtain
non-sharp estimates on the first two Taylor-Maclaurin coefficients $|a_{2}|$
and $|a_{3}|$. Not much is known about the bounds on the general coefficient$%
\ \left\vert a_{n}\right\vert $ for $n\geq 4.$ In the literature, there are
only a few works determining the general coefficient bounds $\left\vert
a_{n}\right\vert $ for the analytic bi-univalent functions (\cite{Altinkaya
and Yalcin 2014a}, \cite{Hamidi and Jahangiri 2014}, \cite{Jahangiri and
Hamidi 2013}). The coefficient estimate problem for each of $\left\vert
a_{n}\right\vert $ $\left( \ n\in 
%TCIMACRO{\U{2115} }%
%BeginExpansion
\mathbb{N}
%EndExpansion
\backslash \left\{ 1,2\right\} ;\ \ 
%TCIMACRO{\U{2115} }%
%BeginExpansion
\mathbb{N}
%EndExpansion
=\left\{ 1,2,3,...\right\} \right) $ is still an open problem.

Chebyshev polynomials have become increasingly important in numerical
analysis, from both theoretical and practical points of view. There are four
kinds of Chebyshev polynomials. The majority of books and research papers
dealing with specific orthogonal polynomials of Chebyshev family, contain
mainly results of Chebyshev polynomials of first and second kinds $T_{n}(t)$
and $U_{n}(t)$ and their numerous uses in different applications, see for
example, Doha \cite{Doha 94} and Mason \cite{Mason 67}.

The Chebyshev polynomials of the first and second kinds are well known. In
the case of a real variable $x$ on $\left( -1,1\right) $, they are defined by

\begin{equation*}
T_{n}(t)=\cos n\theta ,
\end{equation*}

\begin{equation*}
U_{n}(t)=\frac{\sin (n+1)\theta }{\sin \theta },
\end{equation*}%
where the subscript $n$ denotes the polynomial degree and where $t=\cos
\theta $.

If the functions $f$ and $g$ are analytic in $D,$ then $f$ is said to be
subordinate to $g,$ written as%
\begin{equation*}
f\left( z\right) \prec g\left( z\right) ,\ \ \ \ \ \ \ \ \ \ \left( z\in
D\right)
\end{equation*}%
if there exists a Schwarz function $w\left( z\right) ,$ analytic in $D,$ with%
\begin{equation*}
w\left( 0\right) =0~\ \text{and\ \ \ \ \ }\left\vert w\left( z\right)
\right\vert <1\ \ \ \ \ \ \ \ \ \ \left( z\in D\right)
\end{equation*}%
such that 
\begin{equation*}
f\left( z\right) =g\left( w\left( z\right) \right) \ \ \ \ \ \ \ \ \ \ \ \
\left( z\in D\right) .
\end{equation*}

\begin{definition}
A function $f\in \Sigma $ is said to be in the class $H_{\Sigma }\left(
\lambda ,t\right) ,$ $\lambda \geq 0\ $and $t\in \left( \frac{\sqrt{2}}{2},1%
\right] ,$ if the following subordination hold%
\begin{equation}
\left( 1-\lambda \right) \dfrac{zf^{\prime }\left( z\right) }{f\left(
z\right) }+\lambda \left( 1+\dfrac{zf^{\prime \prime }\left( z\right) }{%
f^{\prime }\left( z\right) }\right) \prec H(z,t)=\frac{1}{1-2tz+z^{2}}\ \ \
\ \ \ (z\in D)  \label{eq2}
\end{equation}%
and%
\begin{equation}
\left( 1-\lambda \right) \dfrac{wg^{\prime }\left( w\right) }{g\left(
w\right) }+\lambda \left( 1+\dfrac{wg^{\prime \prime }\left( w\right) }{%
g^{\prime }\left( w\right) }\right) \prec H(w,t)=\frac{1}{1-2tw+w^{2}}\ \ \
\ \ \ (w\in D)  \label{eq3}
\end{equation}%
where $g\left( w\right) =f^{-1}\left( w\right) .$
\end{definition}

We note that if $t=\cos \alpha ,\ \alpha \in \left( -\frac{\pi }{3},\frac{%
\pi }{3}\right) ,$ then%
\begin{equation*}
\begin{array}{ll}
H(z,t) & =\dfrac{1}{1-2tz+z^{2}} \\ 
&  \\ 
& =1+\overset{\infty }{\underset{n=1}{\sum }}\dfrac{\sin (n+1)\alpha }{\sin
\alpha }z^{n}\ \ \ (z\in D).%
\end{array}%
\end{equation*}%
Thus%
\begin{equation*}
H(z,t)=1+2\cos \alpha z+(3\cos ^{2}\alpha -\sin ^{2}\alpha )z^{2}+\cdots \ \
\ (z\in D).
\end{equation*}%
Following see, we write%
\begin{equation*}
H(z,t)=1+U_{1}(t)z+U_{2}(t)z^{2}+\cdots \ \ \ (z\in D,\ \ t\in \left(
-1,1\right) ),
\end{equation*}%
where $U_{n-1}=\dfrac{\sin (n\arccos t)}{\sqrt{1-t^{2}}}\ \ (n\in 
%TCIMACRO{\U{2115} }%
%BeginExpansion
\mathbb{N}
%EndExpansion
)$ are the Chebyshev polynomials of the second kind. Also it is known that 
\begin{equation*}
U_{n}(t)=2tU_{n-1}(t)-U_{n-2}(t),
\end{equation*}%
and%
\begin{equation}
\begin{array}{l}
U_{1}(t)=2t, \\ 
\\ 
U_{2}(t)=4t^{2}-1, \\ 
\\ 
U_{3}(t)=8t^{3}-4t, \\ 
\vdots%
\end{array}
\label{eq4}
\end{equation}%
The Chebyshev polynomials $T_{n}(t),\ t\in \lbrack -1,1],\ $of the first
kind have the generating function of the form%
\begin{equation*}
\overset{\infty }{\underset{n=0}{\sum }}T_{n}(t)z^{n}=\dfrac{1-tz}{%
1-2tz+z^{2}}\ \ \ (z\in D).
\end{equation*}%
However, the Chebyshev polynomials of the first kind $T_{n}(t)$ and the
second kind $U_{n}(t)$ are well connected by the following relationships%
\begin{equation*}
\begin{array}{l}
\dfrac{dT_{n}(t)}{dt}=nU_{n-1}(t), \\ 
\\ 
T_{n}(t)=U_{n}(t)-tU_{n-1}(t), \\ 
\\ 
2T_{n}(t)=U_{n}(t)-U_{n-2}(t).%
\end{array}%
\end{equation*}

In this paper, motivated by the earlier work of Dziok et al. \cite{Dziok
2015}, we use the Chebyshev polynomial expansions to provide estimates for
the initial coefficients of bi-univalent functions in $H_{\Sigma }\left(
\lambda ,t\right) $. We also solve Fekete-Szeg\"{o} problem for functions in
this class.

\section{Coefficient bounds for the function class $H_{\Sigma }\left( 
\protect\lambda ,t\right) $}

\begin{theorem}
Let $\ $the function $f\left( z\right) $ given by (\ref{eq1}) be in the
class $H_{\Sigma }\left( \lambda ,t\right) .$ Then%
\begin{equation*}
\left\vert a_{2}\right\vert \leq \frac{2t\sqrt{2t}}{\sqrt{\left\vert \left(
1+\lambda \right) ^{2}-4\left( \lambda +\lambda ^{2}\right) t^{2}\right\vert 
}}
\end{equation*}%
and%
\begin{equation*}
\left\vert a_{3}\right\vert \leq \frac{4t^{2}}{\left( 1+\lambda \right) ^{2}}%
+\frac{t}{1+2\lambda }.
\end{equation*}
\end{theorem}

\begin{proof}
Let $f\in H_{\Sigma }\left( \lambda ,t\right) .$ From (\ref{eq2}) and (\ref%
{eq3}), we have%
\begin{equation}
\left( 1-\lambda \right) \dfrac{zf^{\prime }\left( z\right) }{f\left(
z\right) }+\lambda \left( 1+\dfrac{zf^{\prime \prime }\left( z\right) }{%
f^{\prime }\left( z\right) }\right) =1+U_{1}(t)w(z)+U_{2}(t)w^{2}(z)+\cdots ,
\label{eq5}
\end{equation}%
and%
\begin{equation}
\left( 1-\lambda \right) \dfrac{wg^{\prime }\left( w\right) }{g\left(
w\right) }+\lambda \left( 1+\dfrac{wg^{\prime \prime }\left( w\right) }{%
g^{\prime }\left( w\right) }\right) =1+U_{1}(t)v(w)+U_{2}(t)v^{2}(w)+\cdots ,
\label{eq6}
\end{equation}%
for some analytic functions $w,v\ $such that $w(0)=v(0)=0$ and $\left\vert
w(z)\right\vert <1,\left\vert v(w)\right\vert <1$ for all $z\in D.$ From the
equalities (\ref{eq5}) and (\ref{eq6}), we obtain that 
\begin{equation}
\left( 1-\lambda \right) \dfrac{zf^{\prime }\left( z\right) }{f\left(
z\right) }+\lambda \left( 1+\dfrac{zf^{\prime \prime }\left( z\right) }{%
f^{\prime }\left( z\right) }\right) =1+U_{1}(t)c_{1}z+\left[
U_{1}(t)c_{2}+U_{2}(t)c_{1}^{2}\right] z^{2}+\cdots ,  \label{eq7}
\end{equation}%
and%
\begin{equation}
\left( 1-\lambda \right) \dfrac{wg^{\prime }\left( w\right) }{g\left(
w\right) }+\lambda \left( 1+\dfrac{wg^{\prime \prime }\left( w\right) }{%
g^{\prime }\left( w\right) }\right) =1+U_{1}(t)d_{1}w+\left[
U_{1}(t)d_{2}+U_{2}(t)d_{1}^{2}\right] w^{2}+\cdots .  \label{eq8}
\end{equation}%
It is fairly well-known that if $\left\vert w(z)\right\vert =\left\vert
c_{1}z+c_{2}z^{2}+c_{3}z^{3}+\cdots \right\vert <1\ $and $\left\vert
v(w)\right\vert =\left\vert d_{1}w+d_{2}w^{2}+d_{3}w^{3}+\cdots \right\vert
<1,\ z,w\in D,\ $then%
\begin{equation*}
\left\vert c_{j}\right\vert \leq 1,\ \ \ \ \forall j\in 
%TCIMACRO{\U{2115} }%
%BeginExpansion
\mathbb{N}
%EndExpansion
.
\end{equation*}

It follows from (\ref{eq7}) and (\ref{eq8}) that 
\begin{equation}
\left( 1+\lambda \right) a_{2}=U_{1}(t)c_{1},  \label{eq9}
\end{equation}%
\begin{equation}
2\left( 1+2\lambda \right) a_{3}-\left( 1+3\lambda \right)
a_{2}^{2}=U_{1}(t)c_{2}+U_{2}(t)c_{1}^{2},  \label{eq10}
\end{equation}%
and%
\begin{equation}
-\left( 1+\lambda \right) a_{2}=U_{1}(t)d_{1},  \label{eq11}
\end{equation}%
\begin{equation}
2\left( 1+2\lambda \right) \left( 2a_{2}^{2}-a_{3}\right) -\left( 1+3\lambda
\right) a_{2}^{2}=U_{1}(t)d_{2}+U_{2}(t)d_{1}^{2}.  \label{eq12}
\end{equation}%
From (\ref{eq9}) and (\ref{eq11}) we obtain%
\begin{equation}
c_{1}=-d_{1}  \label{eq13}
\end{equation}%
and%
\begin{equation}
2\left( 1+\lambda \right) ^{2}a_{2}^{2}=U_{1}^{2}(t)\left(
c_{1}^{2}+d_{1}^{2}\right) .  \label{eq14}
\end{equation}%
By adding (\ref{eq10}) to (\ref{eq12}), we get 
\begin{equation}
\left[ 4\left( 1+2\lambda \right) -2\left( 1+3\lambda \right) \right]
a_{2}^{2}=U_{1}(t)\left( c_{2}+d_{2}\right) +U_{2}(t)\left(
c_{1}^{2}+d_{1}^{2}\right) .  \label{eq16}
\end{equation}%
By using (\ref{eq14}) in equality (\ref{eq16}), we have 
\begin{equation}
\left[ 2\left( 1+\lambda \right) -\frac{2U_{2}(t)}{U_{1}^{2}(t)}\left(
1+\lambda \right) ^{2}\right] a_{2}^{2}=U_{1}(t)\left( c_{2}+d_{2}\right) .
\label{eq17}
\end{equation}%
From (\ref{eq4}) and (\ref{eq17}) we get 
\begin{equation*}
\left\vert a_{2}\right\vert \leq \frac{2t\sqrt{2t}}{\sqrt{\left\vert \left(
1+\lambda \right) ^{2}-4\left( \lambda +\lambda ^{2}\right) t^{2}\right\vert 
}}.
\end{equation*}%
Next, in order to find the bound on $\left\vert a_{3}\right\vert ,$ by
subtracting (\ref{eq12}) from (\ref{eq10}), we obtain%
\begin{equation}
4\left( 1+2\lambda \right) a_{3}-4\left( 1+2\lambda \right)
a_{2}^{2}=U_{1}(t)\left( c_{2}-d_{2}\right) +U_{2}(t)\left(
c_{1}^{2}-d_{1}^{2}\right) .  \label{eq18}
\end{equation}%
Then, in view of (\ref{eq13}) and (\ref{eq14}) , we have from (\ref{eq18}) 
\begin{equation*}
a_{3}=\frac{U_{1}^{2}(t)}{2\left( 1+\lambda \right) ^{2}}\left(
c_{1}^{2}+d_{1}^{2}\right) +\frac{U_{1}(t)}{4\left( 1+2\lambda \right) }%
\left( c_{2}-d_{2}\right) .
\end{equation*}%
Notice that (\ref{eq4}), we get%
\begin{equation*}
\left\vert a_{3}\right\vert \leq \frac{4t^{2}}{\left( 1+\lambda \right) ^{2}}%
+\frac{t}{1+2\lambda }.
\end{equation*}
\end{proof}

\section{Fekete-Szeg\"{o} inequalities for the function class $H_{\Sigma
}\left( \protect\lambda ,t\right) $}

\begin{theorem}
Let $\ f$ given by (\ref{eq1}) be in the class $H_{\Sigma }\left( \lambda
,t\right) \ $and$~\mu \in 
%TCIMACRO{\U{211d} }%
%BeginExpansion
\mathbb{R}
%EndExpansion
.$ Then%
\begin{equation*}
\left\vert a_{3}-\mu a_{2}^{2}\right\vert \leq \left\{ 
\begin{array}{l}
\dfrac{t}{1+2\lambda }; \\ 
\ \ \ \ \ \ \ for\ \left\vert \mu -1\right\vert \leq \tfrac{1}{8\left(
1+2\lambda \right) }\left\vert \frac{\left( 1+\lambda \right) ^{2}}{t^{2}}%
-4\lambda (1+\lambda )\right\vert \\ 
\\ 
\dfrac{8\left\vert 1-\mu \right\vert t^{3}}{\left\vert 4(1+\lambda
)t^{2}-\left( 1+\lambda \right) ^{2}\left( 4t^{2}-1\right) \right\vert }; \\ 
\ \ \ \ \ \ \ for\ \left\vert \mu -1\right\vert \geq \tfrac{1}{8\left(
1+2\lambda \right) }\left\vert \frac{\left( 1+\lambda \right) ^{2}}{t^{2}}%
-4\lambda (1+\lambda )\right\vert%
\end{array}%
\right. .
\end{equation*}
\end{theorem}

\begin{proof}
From (\ref{eq17}) and (\ref{eq18})%
\begin{equation*}
\begin{array}{ll}
a_{3}-\mu a_{2}^{2} & =\left( 1-\mu \right) \dfrac{U_{1}^{3}(t)\left(
c_{2}+d_{2}\right) }{2\left( 1+\lambda \right) U_{1}^{2}(t)-2U_{2}(t)\left(
1+\lambda \right) ^{2}}+\dfrac{U_{1}(t)}{4(1+2\lambda )}\left(
c_{2}-d_{2}\right) \\ 
&  \\ 
& =U_{1}(t)\left[ \left( h\left( \mu \right) +\frac{1}{4\left( 1+2\lambda
\right) }\right) c_{2}+\left( h\left( \mu \right) -\frac{1}{4\left(
1+2\lambda \right) }\right) d_{2}\right]%
\end{array}%
\end{equation*}%
where%
\begin{equation*}
h\left( \mu \right) =\dfrac{U_{1}^{2}(t)\left( 1-\mu \right) }{2\left[
\left( 1+\lambda \right) U_{1}^{2}(t)-U_{2}(t)\left( 1+\lambda \right) ^{2}%
\right] }.
\end{equation*}%
Then, in view of (\ref{eq4}), we conclude that%
\begin{equation*}
\left\vert a_{3}-\mu a_{2}^{2}\right\vert \leq \left\{ 
\begin{array}{ll}
\dfrac{t}{1+2\lambda } & 0\leq \left\vert h\left( \mu \right) \right\vert
\leq \dfrac{1}{4\left( 1+2\lambda \right) } \\ 
4t\left\vert h\left( \mu \right) \right\vert & \left\vert h\left( \mu
\right) \right\vert \geq \dfrac{1}{4\left( 1+2\lambda \right) }.%
\end{array}%
\right.
\end{equation*}%
Taking $\mu =1$ we get

\begin{corollary}
If $\ f\in H_{\Sigma }\left( \lambda ,t\right) ,$ then 
\begin{equation*}
\left\vert a_{3}-a_{2}^{2}\right\vert \leq \dfrac{t}{1+2\lambda }.
\end{equation*}
\end{corollary}
\end{proof}

\end{document}